\documentclass[11pt,letter]{article}%
\usepackage{graphicx}
\usepackage{caption}
\usepackage{subcaption}
\usepackage{color}
\usepackage{amsmath}
\usepackage{amsopn}
\usepackage{amsfonts}
\usepackage{amssymb}
\usepackage{mathtools}
\usepackage{bbm}
\usepackage{bm}
\usepackage{apacite}
\usepackage[margin=1in]{geometry}
\usepackage[authoryear,semicolon]{natbib}%
\providecommand{\U}[1]{\protect\rule{.1in}{.1in}}
\DeclareMathOperator{\diag}{diag}

\newtheorem{theorem}{Theorem}

\newtheorem{corollary}[theorem]{Corollary}

\newtheorem{remark}[theorem]{Remark}

\newenvironment{proof}[1][Proof]{\noindent\textbf{#1.} }{\ \rule{0.5em}{0.5em}}

\DeclareMathOperator{\E}{E}
\DeclareMathOperator{\mujvec}{vec}

\begin{document}

\title{\MakeUppercase{Multilevel maximum likelihood estimation with application to covariance matrices}}
\author{Marie Tur\v{c}i\v{c}ov\'a$^{\ast}$
\and Jan Mandel$^{\dagger}$
\and Kry\v{s}tof Eben$^{\ddagger}$ }
\date{\vspace{-4ex}}
\maketitle

\noindent$^{\ast}$ Institute of Computer Science, Academy of Sciences of the
Czech Republic\newline Pod Vod\'{a}renskou v\v{e}\v{z}\'{\i} 271/2, 182 07
Praha 8, Czech Republic, and \newline Charles University in Prague, Faculty of
Mathematics and Physics\newline Sokolovsk\'{a} 83, Prague 8, 186 75, Czech
Republic\newline turcicova@cs.cas.cz \vskip3mm \noindent$^{\dagger}$
University of Colorado Denver, Denver, CO 80217-3364, USA, and \newline
Institute of Computer Science, Academy of Sciences of the Czech
Republic\newline Pod Vod\'{a}renskou v\v{e}\v{z}\'{\i} 271/2, 182 07 Praha 8,
Czech Republic \newline Jan.Mandel@ucdenver.edu \vskip3mm \noindent
$^{\ddagger}$ Institute of Computer Science, Academy of Sciences of the Czech
Republic \newline Pod Vod\'{a}renskou v\v{e}\v{z}\'{\i} 271/2, 182 07 Praha 8,
Czech Republic \newline eben@cs.cas.cz \vskip3mm


\subsection*{ABSTRACT}

The asymptotic variance of the maximum likelihood estimate is proved to
decrease when the maximization is restricted to a subspace that contains the
true parameter value. Maximum likelihood estimation allows a systematic
fitting of covariance models to the sample, which is important in data
assimilation. The hierarchical maximum likelihood approach is applied to the
spectral diagonal covariance model with different parameterizations of
eigenvalue decay, and to the sparse inverse covariance model with specified
parameter values on different sets of nonzero entries. It is shown
computationally that using smaller sets of parameters can decrease the
sampling noise in high dimension substantially.

\section{INTRODUCTION}

Estimation of large covariance matrices from small samples is an important
problem in many fields, including spatial statistics, genomics, and ensemble
filtering. One of the prominent applications is data assimilation in
meteorology and oceanography, where the dimension of state vector describing
the atmosphere or ocean is in order of millions or larger. Every practically
available sample is a small sample in this context, since a reasonable
approximation of the full covariance can be obtained only with sample size of
the order of the dimension of the problem
%TCIMACRO{\TeXButton{Vershynin-2012-HCS}{\citep{Vershynin-2012-HCS}}}%
%BeginExpansion
\citep{Vershynin-2012-HCS}%
%EndExpansion
. In practice, the sample covariance\footnote{In this paper, by sample
covariance we mean the maximum likelihood estimate of covariance matrix using
the norming constant $N$ as opposed to the unbiased estimate with norming
constant $(N-1).$} is singular and polluted by spurious correlations.
Nevertheless, it carries useful information (e.g. on covariances present in
the actual atmospheric flow) and different techniques can be applied in order
to improve the covariance model and its practical performance.

One common technique is shrinkage, that is, a linear combination of sample
covariance and a~positive definite target matrix, which prevents the
covariance from being singular. The target matrix embodies some prior
information about the covariance; it can be, e.g., unit diagonal or, more
generally, positive diagonal
%TCIMACRO{\TeXButton{Ledoit-2004-WEL}{\citep{Ledoit-2004-WEL}}}%
%BeginExpansion
\citep{Ledoit-2004-WEL}%
%EndExpansion
. See, e.g., \cite{Schafer-2005-SAL} for a survey of such shrinkage
approaches. Shrinkage of sample covariance towards a fixed covariance matrix
based on a specific model and estimated from historical data (called
background covariance) was used successfully in meteorology
%TCIMACRO{\TeXButton{Hamill-2000-HEK,Wang-2008-HE3}{\citep
%{Hamill-2000-HEK,Wang-2008-HE3}}}%
%BeginExpansion
\citep{Hamill-2000-HEK,Wang-2008-HE3}%
%EndExpansion
. This approach is justified as one which combines actual (called
flow-dependent) and long-term average (called climatologic) information on
spatial covariances present in the 3D meteorological fields.

Another approach to improving on the sample covariance matrix is localization
by suppressing long-range spurious correlations, which is commonly done by
multiplying the sample covariance matrix term by term by a gradual cutoff
matrix
%TCIMACRO{\TeXButton{Buehner-2007-SSL,Furrer-2007-EHP}{\citep
%{Buehner-2007-SSL,Furrer-2007-EHP}} }%
%BeginExpansion
\citep{Buehner-2007-SSL,Furrer-2007-EHP}
%EndExpansion
to suppress off-diagonal entries. The extreme case, when only the diagonal is
left, is particularly advantageous in the spectral domain, as the covariance
of a random field in Fourier space is diagonal if and only if the random field
in cartesian geometry is second order stationary, i.e., the covariance between
the values at two points depends only on their distance vector. Alternatively,
diagonal covariance in a wavelet basis provides spatial variability as well
%TCIMACRO{\TeXButton{Pannekoucke-2007-FPW}{\citep{Pannekoucke-2007-FPW}}}%
%BeginExpansion
\citep{Pannekoucke-2007-FPW}%
%EndExpansion
. Spectral diagonal covariance models were successfully used in operational
statistical interpolation in meteorology in spherical geometry
%TCIMACRO{\TeXButton{Parrish-1992-NMC}{\citep{Parrish-1992-NMC}}}%
%BeginExpansion
\citep{Parrish-1992-NMC}%
%EndExpansion
, and versions of Ensemble Kalman Filter (EnKF) were developed which construct
diagonal covariance in Fourier or wavelet space in every update step of the
filter at low cost, and can operate successfully with small ensembles
%TCIMACRO{\TeXButton{Mandel-2010-FFT,Beezley-2011-WEK,Kasanicky-2015-SDE}%
%{\citep{Mandel-2010-FFT,Beezley-2011-WEK,Kasanicky-2015-SDE}}}%
%BeginExpansion
\citep{Mandel-2010-FFT,Beezley-2011-WEK,Kasanicky-2015-SDE}%
%EndExpansion
.

Sparse covariance models, such as the spectral diagonal, allow a compromise
between realistic assumptions and cheap computations. Another covariance model
taking advantage of sparsity is a Gauss-Markov Random Field (GMRF), based on
the fact that conditional independence of variables implies zero corresponding
elements in the inverse of the covariance matrix
%TCIMACRO{\TeXButton{rue2005gaussian}{\citep{rue2005gaussian}}}%
%BeginExpansion
\citep{rue2005gaussian}%
%EndExpansion
, which leads to modeling the covariance as the inverse of a sparse matrix.

However, both spectral diagonal and sparse inverse covariance models have a
large number of parameters, namely all terms of the sparse matrix (up to
symmetry) which are allowed to attain nonzero values. This results in
overfitting and significant sampling noise for small samples. Therefore, it is
of interest to reduce the number of parameters by adopting additional,
problem-dependent assumptions on the true parameter values.

The principal result of this paper is the observation that if parameters are
fitted as the Maximum Likelihood Estimator (MLE) and the additional
assumptions are satisfied by the true parameters, then the estimate using
fewer parameters is asymptotically more accurate, and often very significantly
so even for small samples..

The paper is organized as follows. In Sec.~\ref{sec:MLE}, we provide a brief
statement of MLE and its asymptotic variance. In Sec.~\ref{sec:nested}, we use
the theory of maximum likelihood estimation to prove that for any two nested
subspaces of the parametric space containing the true parameter, the
asymptotic covariance matrix of the MLE is smaller for the smaller parameter
space. These results hold for a general parameter and, in the special case of
MLE for covariance matrices we do not need any invertibility assumption. The
applications to estimation of covariance matrices by spectral diagonal and
GMRF are presented in Sec.~\ref{sec:application}, and Sec.~\ref{sec:comp}
contains computational illustrations. A comparison of the performance of MLE
for parametric models and of related shrinkage estimators is in
Sec.~\ref{sec:shrinkage}.
%A\ discussion of the suitability of the hierarchical model and
%statistical tests, which can determine if a smaller space of parameters is
%appropriate, is in Sect.~\ref{sec:testing-discussion}.

\section{ASYMPTOTIC VARIANCE OF THE MAXIMUM LIKELIHOOD ESTIMATOR}

\label{sec:MLE}

First, we briefly review some standard results for reference. Suppose
$\mathbb{X}_{N}=\left[  \boldsymbol{X}_{1},\ldots,\boldsymbol{X}_{N}\right]  $
is a random sample from a distribution on $\mathbb{R}^{n}$ with density
$f\left(  \boldsymbol{x},\boldsymbol{\theta}\right)  $ with unknown parameter
vector $\boldsymbol{\theta}$ in a parameter space $\Theta\subset\mathbb{R}%
^{p}$. The maximum likelihood estimate $\boldsymbol{\hat{\theta}}_{N}$ of the
true parameter $\boldsymbol{\theta}^{0}$ is defined by maximizing the
likelihood
\[
\boldsymbol{\hat{\theta}}_{N}=\arg\max_{\boldsymbol{\theta}}\mathcal{L}\left(
\boldsymbol{\theta}|\mathbb{X}_{N}\right)  ,\quad\mathcal{L}\left(
\boldsymbol{\theta}|\mathbb{X}_{N}\right)  =\prod\limits_{i=1}^{N}%
\mathcal{L}\left(  \boldsymbol{\theta}|\boldsymbol{X}_{i}\right)
,\quad\mathcal{L}\left(  \boldsymbol{\theta}|\boldsymbol{x}\right)  =f\left(
\boldsymbol{x},\boldsymbol{\theta}\right)  ,
\]
or, equivalently, maximizing the log likelihood%
\begin{equation}
\boldsymbol{\hat{\theta}}_{N}=\arg\max_{\boldsymbol{\theta}}\ell\left(
\boldsymbol{\theta}|\mathbb{X}_{N}\right)  ,\quad\ell\left(
\boldsymbol{\theta}|\mathbb{X}_{N}\right)  =\sum\limits_{i=1}^{N}\ell\left(
\boldsymbol{\theta}|\boldsymbol{X}_{i}\right)  ,\quad\ell\left(
\boldsymbol{\theta}|\boldsymbol{x}\right)  =\log f\left(  \boldsymbol{x}%
,\boldsymbol{\theta}\right)  . \label{eq:MLE-theta}%
\end{equation}

We adopt the usual assumptions that (i) the true parameter $\boldsymbol{\theta
}^{0}$ lies in the interior of $\Theta$, (ii)~the density $f$ determines the
parameter $\boldsymbol{\theta}$ uniquely in the sense that $f(\boldsymbol{x}%
,\boldsymbol{\theta}_{1})=f(\boldsymbol{x},\boldsymbol{\theta}_{2})$ a.s. if
and only if $\boldsymbol{\theta}_{1}=\boldsymbol{\theta}_{2}$, and (iii)
$f\left(  \boldsymbol{x},\boldsymbol{\theta}\right)  $ is a sufficiently
smooth function of $\boldsymbol{x}$ and $\boldsymbol{\theta.}$ Then the error
of the estimate is asymptotically normal,
\begin{equation}
\sqrt{N}(\hat{\boldsymbol{\theta}}_{N}-\boldsymbol{\theta}^{0}%
)\xrightarrow{d}\mathcal{N}_{p}(\boldsymbol{0},Q_{\boldsymbol{\theta}^{0}%
}),\text{ as }N\rightarrow\infty, \label{normalitaMLE}%
\end{equation}
where%
\begin{equation}
Q_{\boldsymbol{\theta}^{0}}=J_{\boldsymbol{\theta}^{0}}^{-1},\quad
J_{\boldsymbol{\theta}^{0}}=\E\left(  \nabla_{\boldsymbol{\theta}}%
\ell(\boldsymbol{\theta}^{0}|\boldsymbol{X})^{\top}\nabla_{\boldsymbol{\theta
}}\ell(\boldsymbol{\theta}^{0}|\boldsymbol{X})\right)  ,\quad\boldsymbol{X}%
\sim f\left(  \boldsymbol{x},\boldsymbol{\theta}^{0}\right)  .
\label{eq:inform-matrix}%
\end{equation}
The matrix $J_{\boldsymbol{\theta}^{0}}$ is called the Fisher information
matrix for the parameterization $\boldsymbol{\theta}^{0}$. Here,
$\boldsymbol{X}$, $\boldsymbol{x}$, and $\boldsymbol{\theta}$ are columns,
while the gradient $\nabla_{\boldsymbol{\theta}}\ell$ of $\ell$ with respect
to the parameter $\boldsymbol{\theta}$ is a row vector, which is compatible
with the dimensioning of Jacobi matrices below. The mean value in
(\ref{eq:inform-matrix}) is taken with respect to $\boldsymbol{X}$, which is
the only random quantity in (\ref{eq:inform-matrix}). Cf., e.g., \cite[Theorem
5.1]{Lehmann-1998-TPE} for details.

\section{NESTED MAXIMUM LIKELIHOOD ESTIMATORS}

\label{sec:nested}

Now suppose that we have an additional information that the true parameter
$\boldsymbol{\theta}^{0}$ lies in a subspace of $\Theta$, which is
parameterized by $k\leq p$ parameters $(\varphi_{1}$,\ldots, $\varphi
_{k})^{\top}=\boldsymbol{\varphi}$. Denote by $\nabla_{\boldsymbol{\varphi}%
}\boldsymbol{\theta}(\boldsymbol{\varphi})$ the $p\times k$ Jacobi matrix with
entries $\frac{\partial\theta_{i}}{\partial\varphi_{j}}$. In the next theorem,
we derive the asymptotic covariance of the maximum likelihood estimator for
$\boldsymbol{\varphi}$,%
\begin{equation}
\hat{\boldsymbol{\varphi}}_{N}=\arg\max_{\boldsymbol{\varphi}}\ell\left(
\boldsymbol{\varphi}|\mathbb{X}_{N}\right)  ,\quad\ell\left(
\boldsymbol{\varphi}|\mathbb{X}_{N}\right)  =\sum\limits_{i=1}^{N}\ell\left(
\boldsymbol{\varphi}|\boldsymbol{X}_{i}\right)  ,\quad\ell\left(
\boldsymbol{\varphi}|\boldsymbol{x}\right)  =\log f\left(  \boldsymbol{x}%
,\boldsymbol{\theta}\left(  \boldsymbol{\varphi}\right)  \right)  ,
\label{eq:MLE-phi}%
\end{equation}
based on the asymptotic covariance of $\boldsymbol{\theta}$ in
(\ref{normalitaMLE}).

\begin{theorem}
\label{thm:sub-asymp}Assume that the map $\boldsymbol{\varphi}\mapsto
\boldsymbol{\theta}(\boldsymbol{\varphi})$ is one-to-one from $\Phi
\subset\mathbb{R}^{k}$ to $\Theta$, the map $\boldsymbol{\varphi}%
\mapsto\boldsymbol{\theta}(\boldsymbol{\varphi})$ is continuously
differentiable, $\nabla_{\boldsymbol{\varphi}}\boldsymbol{\theta
}(\boldsymbol{\varphi})$ is full rank for all $\boldsymbol{\varphi}\in\Phi$,
and $\boldsymbol{\theta}^{0}=\boldsymbol{\theta}(\boldsymbol{\varphi}^{0})$
with $\boldsymbol{\varphi}^{0}$ in the interior of $\Phi$. Then,
\begin{equation}
\sqrt{N}(\hat{\boldsymbol{\varphi}}_{N}-\boldsymbol{\varphi}^{0}%
)\xrightarrow{d}\mathcal{N}_{k}\left(  \boldsymbol{0},Q_{\boldsymbol{\varphi
}^{0}}\right)  \text{ as }N\rightarrow\infty, \label{dist_c_alfa}%
\end{equation}
where $Q_{\boldsymbol{\varphi}^{0}}=J_{\boldsymbol{\varphi}^{0}}^{-1}$, with
$J_{\boldsymbol{\varphi}^{0}}$ the Fisher information matrix of the
parameterization $\boldsymbol{\varphi}$ given by%
\begin{equation}
J_{\boldsymbol{\varphi}^{0}}=\nabla_{\boldsymbol{\varphi}}\boldsymbol{\theta
}(\boldsymbol{\varphi}^{0})^{\top}J_{\boldsymbol{\theta}^{0}}\nabla
_{\boldsymbol{\varphi}}\boldsymbol{\theta}(\boldsymbol{\varphi}^{0}).
\label{J_c_alfa}%
\end{equation}

\end{theorem}

\begin{proof}
From (\ref{eq:inform-matrix}) and the chain rule%
\[
\nabla_{\boldsymbol{\varphi}}\ell(\boldsymbol{\varphi}|\boldsymbol{X}%
)=\nabla_{\boldsymbol{\theta}}\ell(\boldsymbol{\theta}|\boldsymbol{X}%
)\nabla_{\boldsymbol{\varphi}}\boldsymbol{\theta}(\boldsymbol{\varphi}),
\]
we have
\begin{align*}
J_{\boldsymbol{\varphi}^{0}}  &  =\E\left(  \nabla_{\boldsymbol{\varphi}}%
\ell(\boldsymbol{\varphi}^{0}|\boldsymbol{X})^{\top}\nabla
_{\boldsymbol{\varphi}}\ell(\boldsymbol{\varphi}^{0}|\boldsymbol{X})\right) \\
&  =\nabla_{\boldsymbol{\varphi}}\boldsymbol{\theta}(\boldsymbol{\varphi}%
^{0})^{\top}\E\left(  \nabla_{\boldsymbol{\theta}}\ell(\boldsymbol{\theta}%
^{0}|\boldsymbol{X})^{\top}\nabla_{\boldsymbol{\theta}}\ell(\boldsymbol{\theta
}^{0}|\boldsymbol{X})\right)  \nabla_{\boldsymbol{\varphi}}\boldsymbol{\theta
}(\boldsymbol{\varphi}^{0})\\
&  =\nabla_{\boldsymbol{\varphi}}\boldsymbol{\theta}(\boldsymbol{\varphi}%
^{0})^{\top}J_{\boldsymbol{\theta}^{0}}\nabla_{\boldsymbol{\varphi}}
\boldsymbol{\theta}(\boldsymbol{\varphi}^{0}).
\end{align*}
The asymptotic distribution (\ref{dist_c_alfa}) is now (\ref{normalitaMLE})
applied to $\boldsymbol{\varphi}$.
\end{proof}

When the parameter $\boldsymbol{\theta}$ is the quantity of interest in an
application, it is useful to express the estimate and its variance in terms of
the original parameter $\boldsymbol{\theta}$ rather than the subspace
parameter~$\boldsymbol{\varphi}$.

\begin{corollary}
Under the assumptions of Theorem \ref{thm:sub-asymp},%
\begin{equation}
\sqrt{N}(\boldsymbol{\theta}\left(  \hat{\boldsymbol{\varphi}}_{N}\right)
-\boldsymbol{\theta}^{0})\xrightarrow{d}\mathcal{N}_{p}\left(  \boldsymbol{0}%
,Q_{\boldsymbol{\theta}\left(  \boldsymbol{\varphi}^{0}\right)  }\right)
\text{ as }N\rightarrow\infty, \label{distr_phi_c_alfa}%
\end{equation}
where%
\begin{equation}
Q_{\boldsymbol{\theta}\left(  \boldsymbol{\varphi}^{0}\right)  }%
=\nabla_{\boldsymbol{\varphi}}\boldsymbol{\theta}(\boldsymbol{\varphi}%
^{0})J_{\boldsymbol{\varphi}^{0}}^{-1}\nabla_{\boldsymbol{\varphi}%
}\boldsymbol{\theta}(\boldsymbol{\varphi}^{0})^{\top}=\nabla
_{\boldsymbol{\varphi}}\boldsymbol{\theta}(\boldsymbol{\varphi}^{0})\left(
\nabla_{\boldsymbol{\varphi}}\boldsymbol{\theta}(\boldsymbol{\varphi}%
^{0})^{\top}J_{\boldsymbol{\theta}^{0}}\nabla_{\boldsymbol{\varphi}%
}\boldsymbol{\theta}(\boldsymbol{\varphi}^{0})\right)  ^{-1}\nabla
_{\boldsymbol{\varphi}}\boldsymbol{\theta}(\boldsymbol{\varphi}^{0})^{\top}.
\label{eq:J-inv-proj}%
\end{equation}

\end{corollary}

\begin{proof}
The lemma follows from (\ref{dist_c_alfa}) by the delta method \cite[p.~387]%
{Rao-1973-LSI}, since the map $\boldsymbol{\varphi\mapsto\theta}%
(\boldsymbol{\varphi})$ is continuously differentiable.
\end{proof}

\begin{remark}
The matrix $Q_{\boldsymbol{\theta}\left(  \boldsymbol{\varphi}^{0}\right)  }$
is singular, so it cannot be written as the inverse of another matrix, but it
can be understood as the inverse $J_{\boldsymbol{\theta}(\boldsymbol{\varphi
}^{0})}^{-1}$ of the Fisher information matrix for $\boldsymbol{\varphi}$,
embedded in the larger parameter space $\Theta$.
\end{remark}

Suppose that $\boldsymbol{\psi}$ is another parameterization which satisfies
the same assumption as $\boldsymbol{\varphi}$ in Theorem \ref{thm:sub-asymp}:
the map $\boldsymbol{\psi}\mapsto\boldsymbol{\theta}(\boldsymbol{\psi})$ is
one-to-one from $\Psi\subset\mathbb{R}^{m}$, $k \leq m \leq p, $ to $\Theta$,
$\boldsymbol{\psi}\mapsto\boldsymbol{\theta}(\boldsymbol{\psi})$ is
continuously differentiable, $\nabla_{\boldsymbol{\psi}}\boldsymbol{\theta
}(\boldsymbol{\psi})$ is full rank for all $\boldsymbol{\psi}\in\Psi$, and
$\boldsymbol{\theta}^{0}=\boldsymbol{\theta}(\boldsymbol{\psi}^{0})$, where
$\boldsymbol{\psi}^{0}$ is in the interior of $\Psi$. Then, similarly as in
(\ref{distr_phi_c_alfa}), we have also%
\begin{equation}
\sqrt{N}(\boldsymbol{\theta}( \hat{\boldsymbol{\psi}}_{N}) -\boldsymbol{\theta
}^{0})\xrightarrow{d}\mathcal{N}_{p}\left(  \boldsymbol{0}%
,Q_{\boldsymbol{\theta}\left(  \boldsymbol{\psi}^{0}\right)  }\right)  \text{
as }N\rightarrow\infty, \label{distr_psi_c_alfa}%
\end{equation}
where, as in (\ref{eq:J-inv-proj}),%
\begin{equation}
Q_{\boldsymbol{\theta}\left(  \boldsymbol{\psi}^{0}\right)  }=\nabla
_{\boldsymbol{\psi}}\boldsymbol{\theta}(\boldsymbol{\psi}^{0}%
)J_{\boldsymbol{\psi}^{0}}^{-1}\nabla_{\boldsymbol{\psi}}\boldsymbol{\theta
}(\boldsymbol{\psi}^{0})^{\top}=\nabla_{\boldsymbol{\psi}}\boldsymbol{\theta
}(\boldsymbol{\psi}^{0})\left(  \nabla_{\boldsymbol{\psi}}\boldsymbol{\theta
}(\boldsymbol{\psi}^{0})^{\top}J_{\boldsymbol{\theta}^{0}}\nabla
_{\boldsymbol{\psi}}\boldsymbol{\theta}(\boldsymbol{\psi}^{0})\right)
^{-1}\nabla_{\boldsymbol{\psi}}\boldsymbol{\theta}(\boldsymbol{\psi}%
^{0})^{\top}. \label{eq:J-inv-psi-proj}%
\end{equation}

The next theorem shows that when we have two parameterizations
$\boldsymbol{\varphi}$ and $\boldsymbol{\psi}$ which are nested, then the
smaller parameterization has smaller or equal asymptotic covariance than the
larger one. For symmetric matrices $A$ and $B$, $A\leq B$ means that $A-B$ is
positive semidefinite.

\begin{theorem}
\label{thm:comparison}Suppose that $\boldsymbol{\varphi}$ and
$\boldsymbol{\psi}$ satisfy the assumptions in Theorem \ref{thm:sub-asymp},
and there exists a differentiable mapping $\boldsymbol{\varphi\mapsto\psi}$
from $\Phi$ to $\Psi$, such that $\boldsymbol{\varphi}^{0}\boldsymbol{\mapsto
\psi}^{0}$. Then,%
\begin{equation}
Q_{\boldsymbol{\theta}\left(  \boldsymbol{\varphi}^{0}\right)  }\leq
Q_{\boldsymbol{\theta}\left(  \boldsymbol{\psi}^{0}\right)  }.
\label{nerovnost}%
\end{equation}
In addition, if $U\sim\mathcal{N}_{p}\left(  \boldsymbol{0}%
,Q_{\boldsymbol{\theta}\left(  \boldsymbol{\varphi}^{0}\right)  }\right)  $
and $V\sim\mathcal{N}_{p}\left(  \boldsymbol{0},Q_{\boldsymbol{\theta}\left(
\boldsymbol{\psi}^{0}\right)  }\right)  $ are random vectors with the
asymptotic distributions of the estimates $\boldsymbol{\theta}\left(
\hat{\boldsymbol{\varphi}}_{N}\right)  $ and $\boldsymbol{\theta}\left(
\hat{\boldsymbol{\psi}}_{N}\right)  $, then%
\begin{equation}
\E\left\vert U\right\vert ^{2}=\frac{1}{N}\operatorname*{Tr}%
Q_{\boldsymbol{\theta}\left(  \boldsymbol{\varphi}^{0}\right)  }\leq\frac
{1}{N}\operatorname*{Tr}Q_{\boldsymbol{\theta}\left(  \boldsymbol{\psi}%
^{0}\right)  }=\E\left\vert V\right\vert ^{2}, \label{eq:2-norm-comparison}%
\end{equation}
where $\left\vert V\right\vert =\left(  V^{\top}V\right)  ^{1/2}$ is the
standard Euclidean norm in $\mathbb{R}^{p}$.
\end{theorem}

\begin{proof}
Denote $A=J_{\boldsymbol{\theta}^{0}}$, $B=\nabla_{\boldsymbol{\varphi}%
}\boldsymbol{\theta}(\boldsymbol{\varphi}^{0})$, $C=\nabla_{\boldsymbol{\psi}%
}\boldsymbol{\theta}(\boldsymbol{\psi}^{0})$. From the chain rule,%
\[
\nabla_{\boldsymbol{\varphi}}\boldsymbol{\theta}\left(  \boldsymbol{\varphi
}^{0}\right)  =\nabla_{\boldsymbol{\psi}}\boldsymbol{\theta}\left(
\boldsymbol{\psi}^{0}\right)  \nabla_{\boldsymbol{\varphi}}\boldsymbol{\psi
}\left(  \boldsymbol{\varphi}^{0}\right)  ,
\]
we have that $B=C\nabla_{\boldsymbol{\varphi}}\boldsymbol{\psi}\left(
\boldsymbol{\varphi}^{0}\right)  $, and, consequently, $\operatorname*{Range}%
B\subset\operatorname*{Range}C$. Define%
\begin{align*}
P_{B}  &  =A^{1/2}B(B^{\top}AB)^{-1}B^{\top}A^{1/2},\\
P_{C}  &  =A^{1/2}C(C^{\top}AC)^{-1}C^{\top}A^{1/2}.
\end{align*}
The matrices $P_{B}$ and $P_{C}$ are symmetric and idempotent, hence they are
orthogonal projections. In addition,
\[
\operatorname*{Range}P_{B}=\operatorname*{Range}A^{1/2}B\subset
\operatorname*{Range}A^{1/2}C=\operatorname*{Range}P_{C}.
\]
Consequently, $P_{B}\leq P_{C}$ holds from standard properties of orthogonal
projections, and (\ref{nerovnost}) follows.

To prove (\ref{eq:2-norm-comparison}), note that for random vector $X$ with
$\E X=0$ and finite second moment, $\E\left\vert X\right\vert ^{2}%
=\operatorname*{Tr}\operatorname{Cov}X$ from Karhunen-Lo\`{e}ve decomposition
and Parseval identity. The proof is concluded by using the fact that for
symmetric matrices, $A\leq B$ implies $\operatorname*{Tr}A\leq
\operatorname*{Tr}B$, cf. e.g., \cite{Carlen-2010-TIQ}.
\end{proof}

\begin{remark}
In the practically interesting cases when there is a large difference in the
dimensions of the parameters $\boldsymbol{\varphi}$ and $\boldsymbol{\psi}$,
many eigenvalues in the covariance of the estimation error become zero. The
computational tests in Sec.~\ref{sec:comp} show that the resulting decrease of
the estimation error can be significant.
\end{remark}

\section{APPLICATION: NESTED COVARIANCE MODELS}

\label{sec:application}

Models of covariance (e.g., of the state vector in a numerical weather
prediction model) and the quality of the estimated covariance are one of the
key components of data assimilation algorithms. High dimension of the problem
usually prohibits working with the covariance matrix explicitly. In ensemble
filtering methods, this difficulty may be circumvented by working directly
with the original small sample like in the classical Ensemble Kalman filter.
This, however, effectively means using the sample covariance matrix with its
rank deficiency and spurious correlations. Current filtering methods use
shrinkage and localization as noted above, and ad hoc techniques for dimension reduction.

A reliable way towards effective filtering methods lies in introducing
sparsity into covariance matrices or their inverses by means of suitable
covariance models. The results of previous section suggest that it is
beneficial to choose parsimonious models, and indeed, in practical application
we often encounter models with a surprisingly low number of parameters.

A large class of covariance models which encompass sparsity in an efficient
manner arises from Graphical models \citep{lauritzen1996graphical} and
Gaussian Markov Random Fields (GMRF), \citep{rue2005gaussian}, where a special
structure of inverse covariance is assumed. In the area of GMRF, nested
covariance models arise naturally. If, for instance, we consider a GMRF on a
rectangular mesh, each gridpoint may have 4, 8, 12, 20 etc. neighbouring
points which have nonzero corresponding element in the inverse covariance
matrix. Thus, a block band-diagonal structure in the inverse covariance
arises
%TCIMACRO{\TeXButton{Ueno-2009-CRI}{\citep{Ueno-2009-CRI}}}%
%BeginExpansion
\citep{Ueno-2009-CRI}%
%EndExpansion
. The results of Section \ref{sec:nested} apply for this case and we shall
illustrate them in the simulation study of Section \ref{sec:comp}.

Finally, variational assimilation methods, which dominate today's practice of
meteorological services, usually employ a covariance model based on a series
of transformations leading to independence of variables
\citep{Bannister-2008-RFE, Michel-2010-IBE}. At the end, this results in an
estimation problem for normal distribution with a diagonal covariance matrix.

For both ensemble and variational methods, any additional knowledge can be
used to improve the estimate of covariance. Second-order stationarity leads to
diagonality in spectral space, diagonality in wavelet space is often a
legitimate assumption \citep[e.g.,][]{Pannekoucke-2007-FPW} and we shall treat
the diagonal case in more detail.
%We thus come to the estimation problem (with the usual normality assumption)
Suppose
\begin{equation}
\boldsymbol{X}\sim\mathcal{N}_{n}(\boldsymbol{0},D), \label{Xnorm}%
\end{equation}
where $\boldsymbol{X}$ denotes the random field after the appropriate
transform and $D$ is a diagonal matrix. It is clear that estimating $D$ by the
full sample covariance matrix (what would be the case when using the classical
EnKF) is ineffective in this situation and it is natural to use only the
diagonal part of the sample covariance. In practice, the resulting diagonal
matrix may still turn out to be noisy
%TCIMACRO{\TeXButton{Kasanicky-2015-SDE}{\citep{Kasanicky-2015-SDE}}}%
%BeginExpansion
\citep{Kasanicky-2015-SDE}%
%EndExpansion
, and further assumptions like a certain type of decay of the diagonal entries
may be realistic.

In what follows we briefly introduce the particular covariance structures,
state some known facts on full and diagonal covariance, propose parametric
models for the diagonal and compute corresponding MLE.

\subsection{Sample covariance}

\label{sec:sample-covariance}

The top-level parameter space $\Theta$ consists of all symmetric positive
definite matrices, resulting in the parameterization $\Sigma$ with
$\frac{n\left(  n+1\right)  }{2}$ independent parameters. The likelihood of a
sample $\mathbb{X}_{N}=\left[  \boldsymbol{X}^{(1)},\ldots,\boldsymbol{X}%
^{(N)}\right]  $ from $\mathcal{N}_{n}(\boldsymbol{0},\Sigma)$ is%
\[
L\left(  \Sigma|\mathbb{X}_{N}\right)  =\frac{1}{\left(  \det\Sigma\right)
^{N/2}\left(  2\pi\right)  ^{nN/2}}e^{-\frac{1}{2}\operatorname*{Tr}\left(
\Sigma^{-1}\mathbb{X}_{N}\mathbb{X}_{N}^{\top}\right)  }.
\]
If $N\geq n$, it is well known (e.g. \cite{Muirhead-2005-AMS}, p. 83) that the
likelihood is maximized at what we call here sample covariance matrix
\begin{equation}
\hat{\Sigma}_{N}=\frac{1}{N}\sum_{i=1}^{N}\boldsymbol{X}^{(i)}\left(
\boldsymbol{X}^{(i)}\right)  ^{\top}. \label{eq:sample-covariance}%
\end{equation}
The Fisher information matrix of the sample covariance estimator is
\cite[p.~356]{Magnus-2007-MDC}%
\[
J^{(0)}(\mujvec(\Sigma))=\frac{1}{2}\Sigma^{-1}\otimes\Sigma^{-1},
\]
where $\otimes$ stands for the Kronecker product and $\mujvec$ is an operator
that transforms a matrix into a~vector by stacking the columns of the matrix
one underneath the other. This matrix has dimension $n^{2}\times n^{2}$.

\begin{remark}
\label{rem:singular} If $\hat{\Sigma}_{N}$ is singular, $L\left(  \hat{\Sigma
}_{N}|\mathbb{X}_{N}\right)  $ cannot be evaluated because that requires the
inverse of $\hat{\Sigma}$. Also, in this case the likelihood $L\left(
\Sigma|\mathbb{X}_{N}\right)  $ is not bounded above on the set of all
$\Sigma>0$, thus the maximum of $L\left(  \Sigma|\mathbb{X}_{N}\right)  $ does
not exist on that space. To show that, consider an orthonormal change of basis
so that the vectors in $\operatorname*{span}\left(  \mathbb{X}_{N}\right)  $
come first, write vectors and matrices in the corresponding $2\times2$ block
form, and let%
\[
\tilde{\Sigma}_{N}=\left[
\begin{array}
[c]{cc}%
\tilde{\Sigma}_{11} & 0\\
0 & 0
\end{array}
\right]  ,\quad\tilde{\Sigma}_{11}>0.
\]
Then $\lim_{a\rightarrow0^{+}}\mathbb{X}_{N}^{\top}\left(  \tilde{\Sigma}%
_{N}+aI\right)  ^{-1}\mathbb{X}_{N}$ exists, but $\lim_{a\rightarrow0^{+}}%
\det\left(  \tilde{\Sigma}_{N}+aI\right)  =0$, thus
\[
\lim_{a\rightarrow0^{+}}L\left(  \tilde{\Sigma}_{N}+aI\right\vert
\mathbb{X}_{N})=\infty.
\]

Note that when the likelihood is redefined in terms of the subspace
$\operatorname*{span}\left(  \mathbb{X}_{N}\right)  $ only, the sample
covariance can be obtained by maximization on the subspace \cite[p.
527]{Rao-1973-LSI}.
\end{remark}

When the true covariance is diagonal ($\Sigma\equiv D$, cf. (\ref{Xnorm})), a
significant improvement can be achieved by setting the off-diagonal terms of
sample covariance to zero,
\begin{equation}
\hat{D}_{N}^{\left(  0\right)  }=\operatorname{diag}\left(  \hat{\Sigma}%
_{N}\right)  . \label{eq:diag-sample-covariance}%
\end{equation}
It is known that using only the diagonal of the unbiased sample covariance
\[
\hat{\Sigma}_{N}^{u}=\frac{1}{N-1}\sum_{i=1}^{N}\boldsymbol{X}^{(i)}\left(
\boldsymbol{X}^{(i)}\right)  ^{\top}%
\]
results in smaller (or equal) Frobenius norm of the error pointwise,%
\begin{equation}
\left\vert \hat{D}_{N}^{\left(  0\right)  }-D\right\vert _{F}\leq\left\vert
\hat{\Sigma}_{N}^{u}-D\right\vert _{F} \label{eq:pointwise}%
\end{equation}
cf. \cite{Furrer-2007-EHP} for the case when the mean is assumed to be known
like here, and \cite{Kasanicky-2015-SDE} for the unbiased sample covariance
and unknown mean.

\subsection{Diagonal covariance}

\label{sec:diag-covariance}

The parameter space $\Theta_{1}$ consisting of all diagonal matrices with
positive diagonal, with $n$ parameters $\boldsymbol{d}=\left(  d_{1}%
,\ldots,d_{n}\right)  ^{\top}$, can be viewed as a simple class of models for
either covariance or its inverse. The log-likelihood function for
$D=\diag(d_{1},\ldots,d_{n})$ with a given random sample $\mathbb{X}%
_{N}=\left[  \boldsymbol{X}^{(1)},\ldots,\boldsymbol{X}^{(N)}\right]  $ from
$\mathcal{N}_{n}\left(  \boldsymbol{0},D\right)  $ is
\[
\ell(D|\mathbb{X}_{N})=-\frac{N}{2}\log\left(  (2\pi)^{n}|D|\right)  -\frac
{1}{2}\sum_{k=1}^{N}\left(  \boldsymbol{X}^{(k)}\right)  ^{\top}%
D^{-1}\boldsymbol{X}^{(k)}%
\]
and has its maximum at
\begin{equation}
\hat{d}_{j}=\frac{1}{N}\sum_{k=1}^{N}\left(  X_{j}^{(k)}\right)  ^{2}%
,\hspace{3mm}j=1,\ldots,n, \label{estD1}%
\end{equation}
where $X_{j}^{(k)}$ denotes the $j$-th entry of $\boldsymbol{X}^{(k)}$. The
sum of squares $S_{j}^{2}=\sum_{k=1}^{N}\left(  X_{j}^{(k)}\right)  ^{2}$ is a
sufficient statistic for the variance $d_{j}$. Thus, we get the maximum
likelihood estimator
\begin{equation}
\hat{D}_{N}^{\left(  1\right)  }=\frac{1}{N}\diag\left(  S_{1}^{2}%
,\ldots,S_{n}^{2}\right)  . \label{D1}%
\end{equation}

It is easy to compute the Fisher information matrix explicitly,%
\begin{equation}
J_{D^{\left(  1\right)  }}=\diag\left(  \frac{1}{2d_{1}^{2}},\ldots,\frac
{1}{2d_{n}^{2}}\right)  . \label{eq:J1}%
\end{equation}
which is an $n \times n$ matrix and gives the asymptotic covariance of the
estimation error%
\[
\frac{1}{N}Q_{D^{\left(  1\right)  }}=\frac{1}{N}J_{D^{\left(  1\right)  }%
}^{-1}=\frac{1}{N}\diag\left(  2d_{1}^{2},\ldots, 2d_{n}^{2}\right)
\]
from (\ref{normalitaMLE}).

\subsection{Diagonal covariance with prescribed decay by 3 parameters}

\label{sec:3-param-covariance}

A more specific situation appears when we have an additional information that
the matrix $D$ is not only diagonal, but its diagonal entries have a
prescribed decay. For instance, this decay can be governed by a model of the
form $d_{i}=((c_{1}+c_{2}h_{i})f_{i}(\alpha))^{-1}$, $i=1,\ldots,n$, where
$c_{1},c_{2}$ and $\alpha$ are unknown parameters, $h_{1},\ldots,h_{n}$ are
known positive numbers, and $f_{1},\ldots,f_{n}$ are known differentiable
functions. For easier computation it is useful to work with $\tau_{i}=\frac
{1}{d_{i}}=(c_{1}+c_{2}h_{i})f_{i}(\alpha)$. Maximum likelihood estimators for
$c_{1},c_{2}$, and $\alpha$ can be computed effectively from the likelihood
\begin{equation}
\ell(D|\mathbb{X}_{N})=-\frac{N}{2}n\log(2\pi)+\frac{N}{2}\sum_{i=1}^{n}%
\log\tau_{i}-\frac{1}{2}\sum_{i=1}^{n}\tau_{i}S_{i}^{2} \label{l_podle_tau}%
\end{equation}
by using the chain rule. It holds that
\begin{align*}
\frac{\partial\ell}{\partial c_{1}}  &  =\sum_{i=1}^{n}\frac{\partial\ell
}{\partial\tau_{i}}\frac{\partial\tau_{i}}{\partial c_{1}}=\sum_{i=1}%
^{n}\left(  \frac{N}{2\tau_{i}}-\frac{S_{i}^{2}}{2}\right)  \frac{\partial
\tau_{i}}{\partial c_{1}}\\
&  =\frac{N}{2}\sum_{i=1}^{n}\left(  \frac{1}{(c_{1}+c_{2}h_{i})f_{i}(\alpha
)}-\frac{1}{N}S_{i}^{2}\right)  f_{i}(\alpha).
\end{align*}
Setting this derivative equal to zero we get
\begin{equation}
\sum_{i=1}^{n}\left(  \frac{1}{c_{1}+c_{2}h_{i}}-\frac{1}{N}S_{i}^{2}%
f_{i}(\alpha)\right)  =0. \label{rceC1}%
\end{equation}
Analogously,
\[
\frac{\partial\ell}{\partial c_{2}}=\sum_{i=1}^{n}\frac{\partial\ell}%
{\partial\tau_{i}}\frac{\partial\tau_{i}}{\partial c_{2}}=\frac{N}{2}%
\sum_{i=1}^{n}\left(  \frac{1}{(c_{1}+c_{2}h_{i})f_{i}(\alpha)}-\frac{1}%
{N}S_{i}^{2}\right)  h_{i}f_{i}(\alpha),
\]
so the equation for estimating the parameter $c_{2}$ is%
\begin{equation}
\sum_{i=1}^{n}\left(  \frac{h_{i}}{c_{1}+c_{2}h_{i}}-\frac{1}{N}S_{i}^{2}%
h_{i}f_{i}(\alpha)\right)  =0. \label{rceC2}%
\end{equation}
Similarly,
\begin{align*}
\frac{\partial\ell}{\partial\alpha}  &  =\sum_{i=1}^{n}\frac{\partial\ell
}{\partial\tau_{i}}\frac{\partial\tau_{i}}{\partial\alpha}=\frac{N}{2}%
\sum_{i=1}^{n}\left(  \frac{1}{(c_{1}+c_{2}h_{i})f_{i}(\alpha)}-\frac{1}%
{N}S_{i}^{2}\right)  (c_{1}+c_{2}h_{i})\frac{\partial f_{i}(\alpha)}%
{\partial\alpha}\\
&  =\frac{N}{2}\sum_{i=1}^{n}\left(  \frac{1}{f_{i}(\alpha)}-\frac{1}{N}%
S_{i}^{2}(c_{1}+c_{2}h_{i})\right)  \frac{\partial f_{i}(\alpha)}%
{\partial\alpha}%
\end{align*}
and setting the derivative to zero, we get
\begin{equation}
\sum_{i=1}^{n}\left(  \frac{1}{f_{i}(\alpha)}\frac{\partial f_{i}(\alpha
)}{\partial\alpha}-\frac{1}{N}S_{i}^{2}(c_{1}+c_{2}h_{i})\frac{\partial
f_{i}(\alpha)}{\partial\alpha}\right)  =0. \label{rceAlfa3}%
\end{equation}

The maximum likelihood estimator for $D$ is then given by
\begin{equation}
\hat{D}^{(3)}=\diag\{((\hat{c}_{1}+\hat{c}_{2}h_{i})f_{i}(\hat{\alpha}))^{-1},
i = 1,\ldots,n \}, \label{D03}%
\end{equation}
where $(\hat{c}_{1},\hat{c}_{2},\hat{\alpha})$ is the solution of the system
(\ref{rceC1}), (\ref{rceC2}), (\ref{rceAlfa3}). This expression corresponds to
searching a maximum likelihood estimator of $D$ in the subspace $\Theta
_{3}\subset\Theta_{1}\subset\Theta$ formed by diagonal matrices
$\diag\{((c_{1}+c_{2}h_{i})f_{i}(\alpha))^{-1},i=1,\ldots,n\}$.

For completeness, the asymptotic covariance of the estimation error about
\[
D^{(3)}=\diag\{d_{i}(c_{1},c_{2},\alpha), i=1,\ldots,n\},
\]
contained in $\mathbb{X}_{N}$ is
\begin{equation}
\frac{1}{N}Q_{D^{\left(  3\right)  }}=\frac{1}{N}\nabla\boldsymbol{d}%
(c_{1},c_{2},\alpha)J_{c_{1},c_{2},\alpha}^{-1}\nabla\boldsymbol{d}%
(c_{1},c_{2},\alpha)^{\top} \label{J_c1_c2_alfa}%
\end{equation}
from (\ref{eq:J-inv-proj}), where the Fisher information matrix $J_{c_{1}%
,c_{2},\alpha}$ is the $3\times3$ matrix%
\begin{align*}
&  J_{c_{1},c_{2},\alpha}=\\
&
\begin{bmatrix}
\frac{1}{2}\sum_{i=1}^{n}\frac{1}{(c_{1}+c_{2}h_{i})^{2}} & \frac{1}{2}%
\sum_{i=1}^{n}\frac{h_{i}}{(c_{1}+c_{2}h_{i})^{2}} & \frac{1}{2}\sum_{i=1}%
^{n}\frac{1}{(c_{1}+c_{2}h_{i})f_{i}(\alpha)}\frac{\partial f_{i}(\alpha
)}{\partial\alpha}\\[0.3em]%
\frac{1}{2}\sum_{i=1}^{n}\frac{h_{i}}{(c_{1}+c_{2}h_{i})^{2}} & \frac{1}%
{2}\sum_{i=1}^{n}\frac{h_{i}^{2}}{(c_{1}+c_{2}h_{i})^{2}} & \frac{1}{2}%
\sum_{i=1}^{n}\frac{h_{i}}{(c_{1}+c_{2}h_{i})f_{i}(\alpha)}\frac{\partial
f_{i}(\alpha)}{\partial\alpha}\\[0.3em]%
\frac{1}{2}\sum_{i=1}^{n}\frac{1}{(c_{1}+c_{2}h_{i})f_{i}(\alpha)}%
\frac{\partial f_{i}(\alpha)}{\partial\alpha} & \frac{1}{2}\sum_{i=1}^{n}%
\frac{h_{i}}{(c_{1}+c_{2}h_{i})f_{i}(\alpha)}\frac{\partial f_{i}(\alpha
)}{\partial\alpha} & \frac{1}{2}\sum_{i=1}^{n}\frac{1}{f_{i}^{2}(\alpha
)}\left(  \frac{\partial f_{i}(\alpha)}{\partial\alpha}\right)  ^{2}%
\end{bmatrix}
\end{align*}
and
\begin{align*}
\boldsymbol{d}(c_{1},c_{2},\alpha)  &  =\left[  d_{1}(c_{1},c_{2}%
,\alpha),\ldots,d_{n}(c_{1},c_{2},\alpha)\right]  ^{\top}\\
&  =\left[  ((c_{1}+c_{2}h_{1})f_{1}(\alpha))^{-1},\ldots,((c_{1}+c_{2}%
h_{n})f_{n}(\alpha))^{-1}\right]  ^{\top}.
\end{align*}

\subsection{Diagonal covariance with prescribed decay by 2 parameters}

\label{sec:2-param-covariance}

We may consider a more specific model for diagonal elements with two
parameters: $d_{i}=(cf_{i}(\alpha))^{-1}$, i.e. $\tau_{i}=cf_{i}(\alpha)$,
$i=1,\ldots,n$, where $c$ and $\alpha$ are unknown parameters. Maximum
likelihood estimators for $c$ and $\alpha$ can be computed similarly as in the
previous case. The estimating equations have the form
\begin{align*}
\frac{1}{c}  &  =\frac{1}{n}\sum_{i=1}^{n}\frac{1}{N}S_{i}^{2}f_{i}(\alpha)\\
\frac{1}{c}\sum_{i=1}^{n}\frac{1}{f_{i}(\alpha)}\frac{\partial f_{i}(\alpha
)}{\partial\alpha}  &  =\sum_{i=1}^{n}\frac{1}{N}S_{i}^{2}\frac{\partial
f_{i}(\alpha)}{\partial\alpha},
\end{align*}
which can be rearranged to
\begin{align}
\frac{1}{c}  &  =\frac{1}{n}\sum_{i=1}^{n}\frac{1}{N}S_{i}^{2}f_{i}%
(\alpha)\label{rceC}\\
0  &  =\sum_{i=1}^{n}S_{i}^{2}f_{i}(\alpha)\left(  \frac{1}{f_{i}(\alpha
)}\frac{\partial f_{i}(\alpha)}{\partial\alpha}-\frac{1}{n}\sum_{j=1}^{n}%
\frac{1}{f_{j}(\alpha)}\frac{\partial f_{j}(\alpha)}{\partial\alpha}\right)  .
\label{rceAlfa2}%
\end{align}
Equation (\ref{rceAlfa2}) is an implicit formula for estimating $\alpha$. Its
result can be used for estimating $c$ through (\ref{rceC}). The maximum
likelihood estimator for $D$ is then given by
\begin{equation}
\hat{D}^{(2)}=\diag\left(  (\hat{c}f_{1}(\hat{\alpha}))^{-1},\ldots,(\hat
{c}f_{n}(\hat{\alpha}))^{-1} \right)  , \label{D02}%
\end{equation}
where $\hat{c}$ and $\hat{\alpha}$ are MLEs of $c$ and $\alpha$. It
corresponds to searching a maximum likelihood estimator of $D$ in the subspace
$\Theta_{2}\subset\Theta_{3}\subset\Theta_{1}\subset\Theta$ formed by diagonal
matrices $\diag\left\{  (cf_{i}(\alpha))^{-1},i=1,\ldots,n \right\}  $. Of
course, the estimator $\hat{D}^{(2)}$ does not have ``larger" variance than
$\hat{D}^{(3)}$.

The covariance of the asymptotic distribution of the parameters $d_{1}%
,\ldots,d_{n}$ is
\begin{equation}
\frac{1}{N}Q_{D^{\left(  2\right)  }}=\frac{1}{N}\nabla\boldsymbol{d}%
(c,\alpha)J_{c,\alpha}^{-1}\nabla\boldsymbol{d}(c,\alpha)^{\top},
\label{J_c_alfa_2}%
\end{equation}
from (\ref{eq:J-inv-proj}), where Fisher information matrix at $D=\diag\{d_{i}%
(c,\alpha),i=1,\ldots,n\}$ is the $2\times2$ matrix
\[
J_{c,\alpha}=%
\begin{bmatrix}
\frac{n}{2c^{2}} & \frac{1}{2c}\sum_{i=1}^{n}\frac{1}{f_{i}(\alpha)}%
\frac{\partial f_{i}(\alpha)}{\partial\alpha}\\[0.3em]%
\frac{1}{2c}\sum_{i=1}^{n}\frac{1}{f_{i}(\alpha)}\frac{\partial f_{i}(\alpha
)}{\partial\alpha} & \frac{1}{2}\sum_{i=1}^{n}\frac{1}{f_{i}^{2}(\alpha
)}\left(  \frac{\partial f_{i}(\alpha)}{\partial\alpha}\right)  ^{2}%
\end{bmatrix}
\]
and $\boldsymbol{d}(c,\alpha)=\left[  d_{1}(c,\alpha),\ldots,d_{n}%
(c,\alpha)\right]  ^{\top}=\left[  (cf_{1}(\alpha))^{-1},\ldots,(cf_{n}%
(\alpha))^{-1}\right]  ^{\top}.$

\subsection{Sparse inverse covariance and GMRF}

\label{sec:GMRF} In the GMRF\ method for fields on a rectangular mesh, we
assume that a variable on a gridpoint is conditionally independent on the rest
of the gridpoints, given values on neighboring gridpoints. It follows that
nonzero entries in the inverse of the covariance matrix can be only between
neighbor gridpoints. We start with 4 neighbors (up, down, right, left), and
adding neighbors gives rise to a sequence of nested covariance models. If the
columns of the mesh are stacked vertically, their inverse covariance matrix
will have a band-diagonal structure.

The inverse covariance model fitted by MLE was introduced by
\cite{Ueno-2009-CRI} and applied on data from oceanography. The corresponding
Fisher information matrix may be found as the negative of the Hessian matrix
\cite[eq.~(C17)]{Ueno-2009-CRI}.

% fig 1
\begin{figure}[h]%
\begin{tabular}
[c]{ccc}%
\includegraphics[width=.3\textwidth]{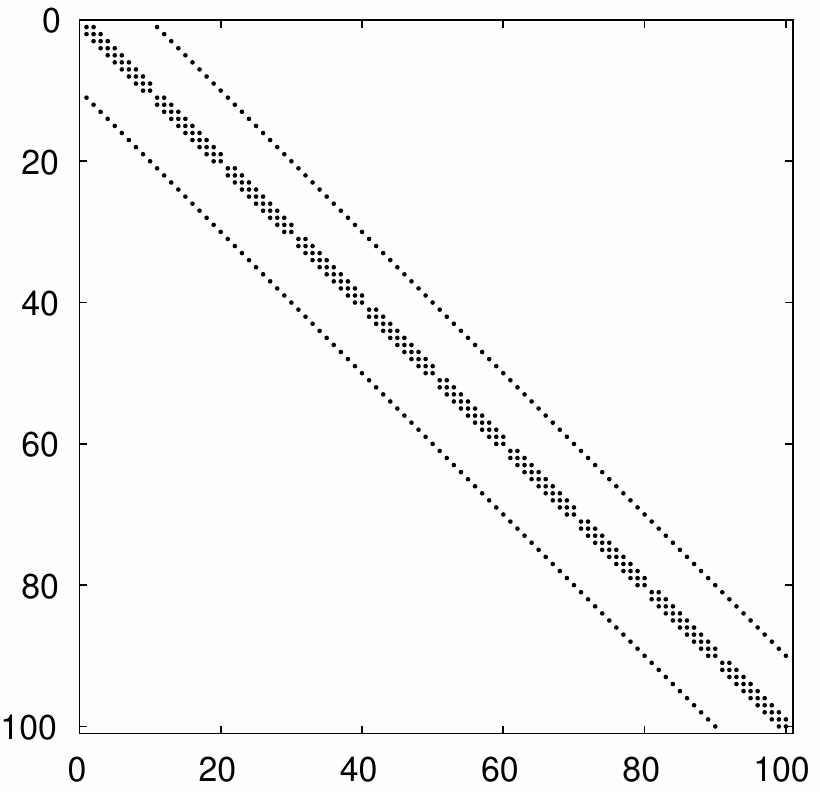} &
\includegraphics[width=.3\textwidth]{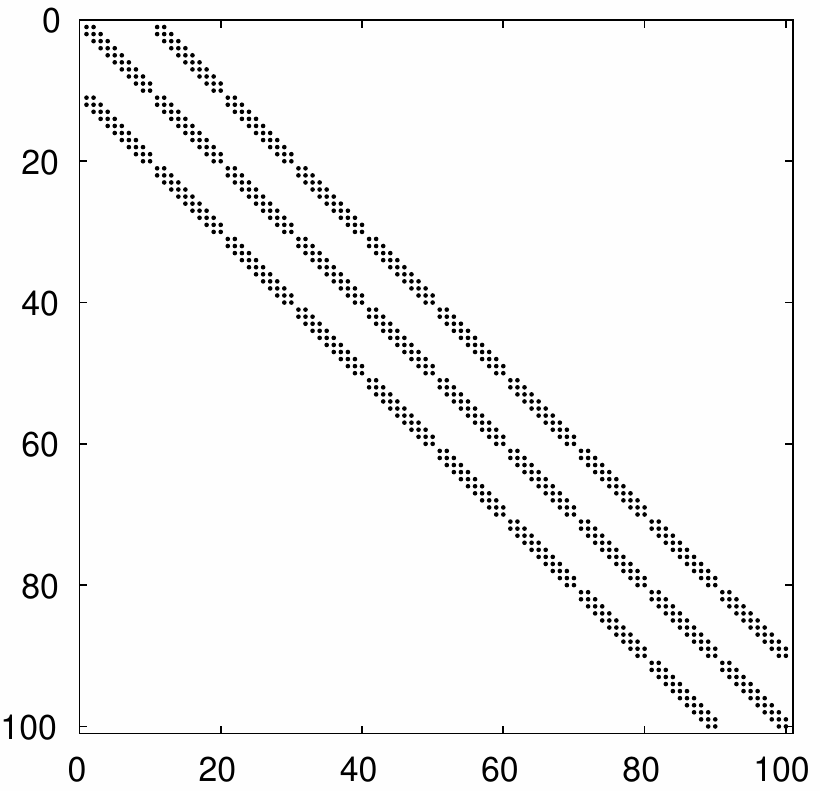} &
\includegraphics[width=.3\textwidth]{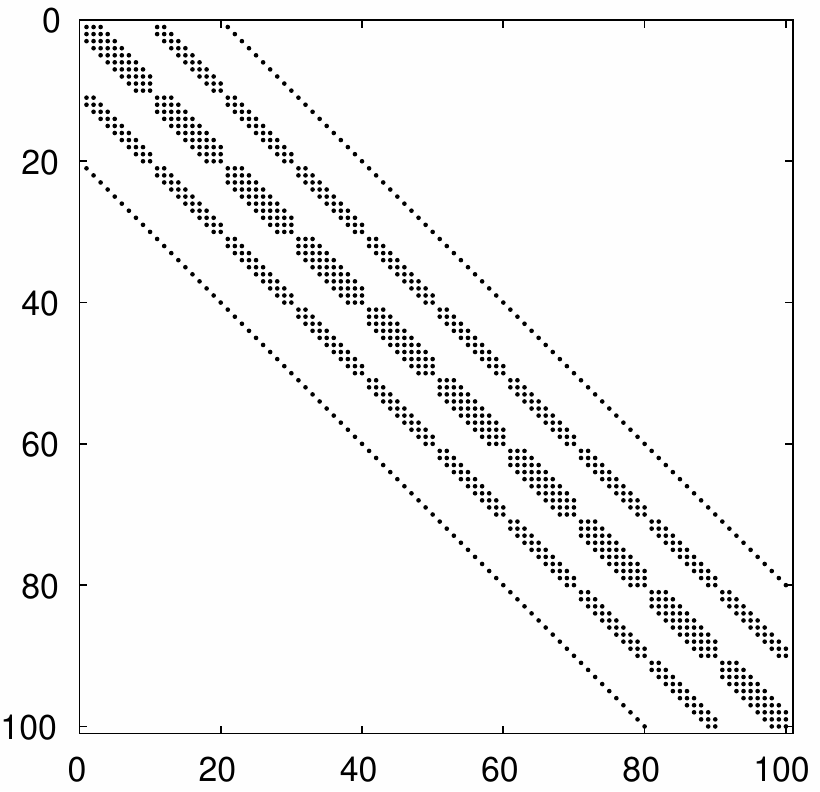}\\
\includegraphics[width=.08\textwidth]{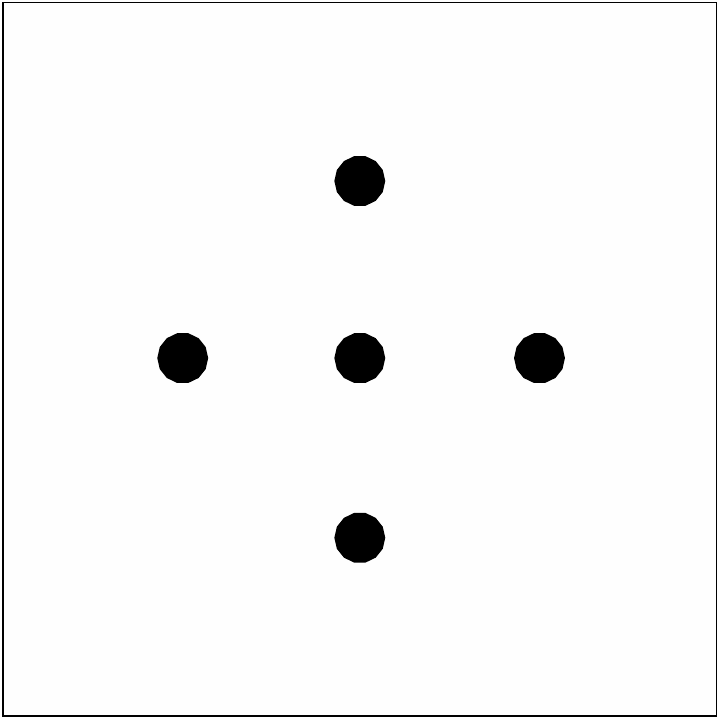} &
\includegraphics[width=.08\textwidth]{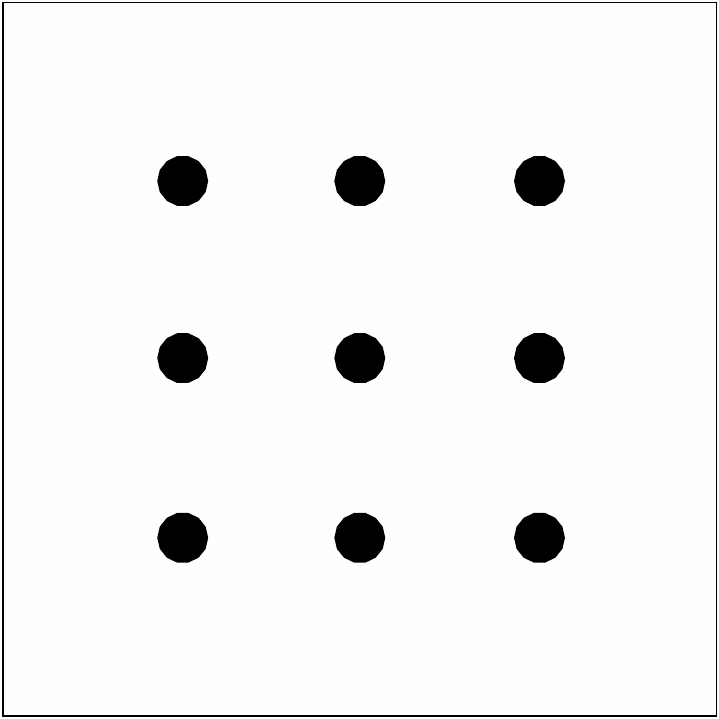} &
\includegraphics[width=.08\textwidth]{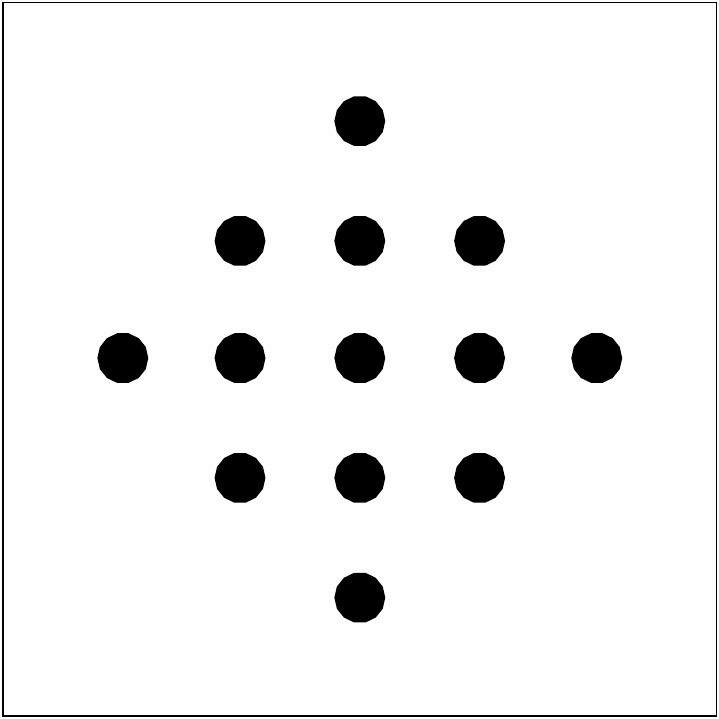}
\end{tabular}
\caption{Block band-diagonal structure of inverse covariance matrix. 10
columns of dimension 10, stacked vertically. 4, 8, 12 neighbors of any
gridpoint. }%
\label{fig:structure}%
\end{figure}

\section{COMPUTATIONAL STUDY}

\label{sec:comp}

In Section \ref{sec:nested}, we have shown that in the sense of asymptotic
variance and second moment (mean-squared) error, the maximum likelihood
estimator computed in a smaller space containing the true parameter is more
(or equally) precise. For small samples, we illustrate this behavior by means
of simulations.

%In both cases, the simulation was carried out as follows. The true covariance structure $\Sigma$ of a field with dimension $10 \times 10$ (resulting in $n=100$) was prepared and then random samples with were ge

\subsection{Simulation of simple GMRF}

\label{sec:GMRF-sim} We first show that in the case of GMRF with four
neighbors per gridpoint, adding dependencies (parameters) which are not
present brings a loss of precision of the MLE. Using the sample covariance in
this case causes a substantial error.

We have generated an ensemble of realizations of a GMRF with dimensions
$10\times10$ (resulting in $n=100$) and inverse covariance structure as in
Fig.~\ref{fig:structure}. The values on the diagonals of the covariance matrix
have been set to constant, since we assume the correlation with left and right
neighbor to be identical, as well as the correlation with upper and lower
neighbor (by symmetry of the covariance matrix and isotropy in both directions
of the field, but different correlation in each direction). This leads to a
model with 3 parameters for 4 neighbors, 5 parameters for 8 neighbors and 7
parameters for 12 neighbors,

The covariance structure of $\Sigma^{-1}$ with 4 neighbors was set as
\textquotedblleft truth\textquotedblright\ and random samples were generated
from $\mathcal{N}_{n}(\boldsymbol{0},\Sigma)$ with sample sizes
$N=10,15,20,\ldots,55$. The values on first, second and tenth diagonal have
been set as 5, -0.2 and 0.5. For each sample, we computed successively the MLE
with 3, 5 and 7 unknown parameters numerically by Newton's method, as
described in \cite{Ueno-2009-CRI}.

The difference of each estimator from the true matrix $\Sigma$ was measured in
the Frobenius norm, which is the same as the Euclidean norm of a matrix
written as one long vector. In order to reduce the sampling error, 50
simulations of the same size were generated and the mean of squared Frobenius
norm was computed. The results can be found in Fig.~\ref{fig:errplot1}.

%For each sample size, 50 replications have been done. The results are plotted in  Fig.~\ref{fig:errplot1}.

%fig 2
\begin{figure}[h]
\begin{center}
%\begin{tabular}
%[c]{cc}%
%\hspace*{-0.5cm}
%\includegraphics[width=.53\textwidth,trim={0cm 0 0 0},clip]{./Simulace/errplot1.pdf} &
%\hspace*{-0.5cm}
%\includegraphics[width=.95\textwidth,trim={0cm 0 0 0},clip]{./Simulace/errplot_samplecov.png}
\includegraphics[height=3in]{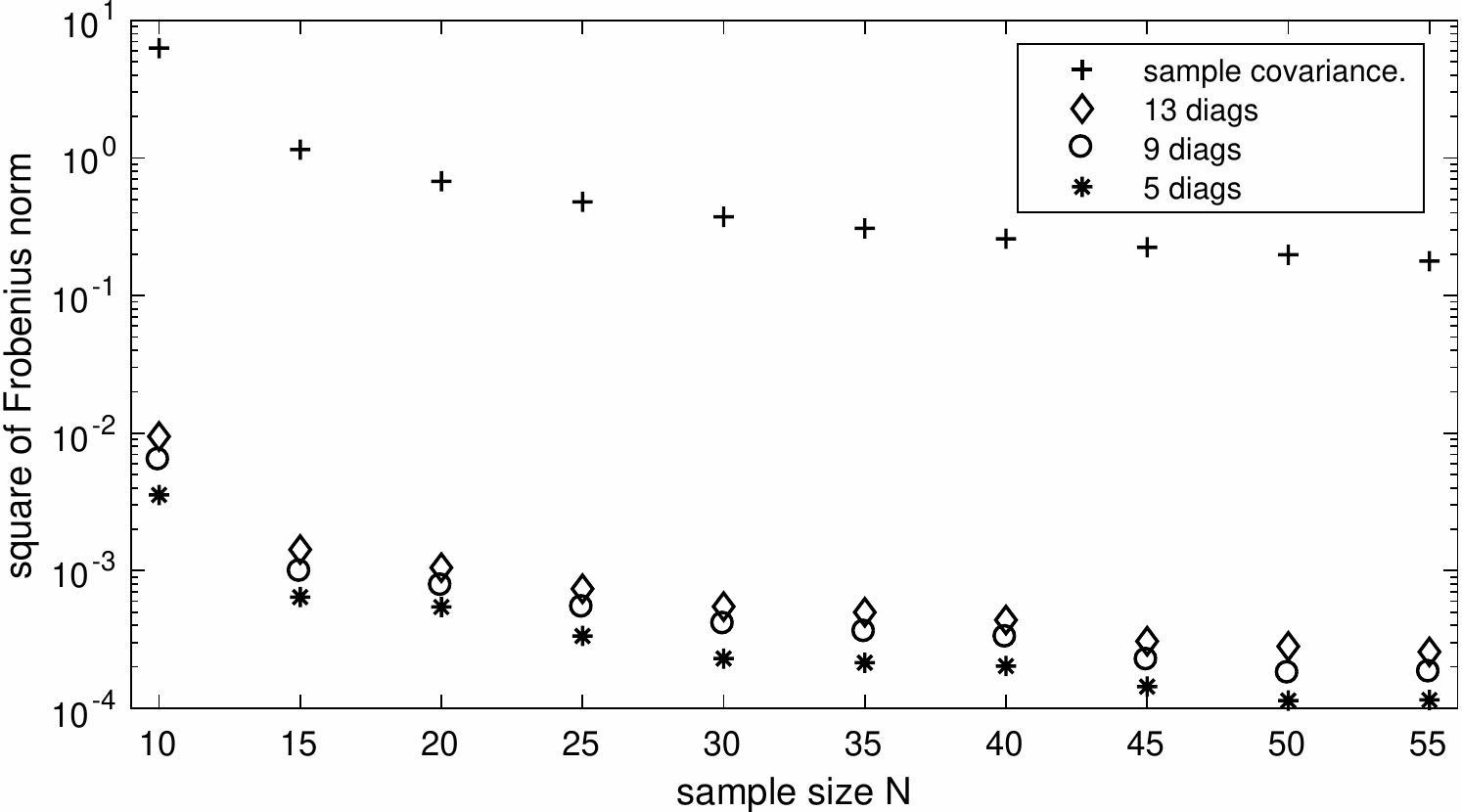}
%\end{tabular}
\end{center}
\caption{Error of the MLE in Frobenius norm for sample covariance and models
with 4, 8, 12 neighbors, i.e. 5, 9, 13 nonzero diagonals in the inverse
covariance matrix. }%
\label{fig:errplot1}%
\end{figure}

As expected, the MLE with 3 parameters outperforms the estimates with 5 and 7
parameters and the Frobenius norm for sample covariance stays one order worse
than all parametric estimates.

\subsection{Simulation of fields with diagonal covariance}

\label{sec:diag-sim}

The simulation for spectral diagonal covariance was carried out in a similar
way. First, a diagonal matrix $D$ was prepared, whose diagonal entries decay
according to the model $d_{i}=\frac{1}{c}e^{\alpha\lambda_{i}},i=1,\ldots,n$,
where $c$ and $\alpha$ are parameters and $\lambda_{i}$ are the eigenvalues of
Laplace operator in two dimensions on $10\times10$ nodes (so again $n=100$).
Such models are useful in modeling smooth random fields, e.g., in meteorology.
Then, random samples were generated from $\mathcal{N}_{n}(\boldsymbol{0},D)$
with sample sizes $N=5,\ldots,20$. For each sample, four covariance matrix
estimators were computed:

\begin{itemize}
\item sample covariance matrix $\hat{\Sigma}_{N}$, cf.
(\ref{eq:sample-covariance})

\item diagonal part $\hat{D}^{\left(  0\right)  }$ of the sample covariance
matrix, cf. (\ref{eq:diag-sample-covariance})

\item MLE $\hat{D}^{\left(  1\right)  }$ in the space of diagonal matrices,
cf. (\ref{D1})

\item MLE $\hat{D}^{\left(  3\right)  }=\diag\{(\hat{c}_{1}-\hat{c}_{2}%
\lambda_{i})^{-1}e^{\hat{\alpha}\lambda_{i}},i=1,\ldots,n\}$ with 3 parameters
$c_{1},c_{2}$ and $\alpha$, cf. (\ref{D03}).

\item MLE $\hat{D}^{\left(  2\right)  }=\diag\{\hat{c}^{-1}e^{\hat{\alpha
}\lambda_{i}},i=1,\ldots,n\}$ with 2 parameters $c$ and $\alpha$, cf.
(\ref{D02}).
\end{itemize}

Let us briefly discuss the choice of the covariance model $d_{i}=\frac{1}%
{c}e^{\alpha\lambda_{i}}$. We decided to carry out the simulation with a
second-order stationary random field, whose covariance can be diagonalized by
the Fourier transform. This transform is formed by the eigenvectors of the
Laplace operator. Hence, it is reasonable to model the diagonal terms of this
covariance matrix (i.e. the covariance eigenvalues) by some function of
eigenvalues of the Laplace operator. This function needs to have a
sufficiently fast decay in order to fulfil the necessary condition for the
proper covariance (the so-called trace class property, e.g.,
\cite{Kuo-1975-GMB}). Exponential decay is used, e.g., in
\cite{Mirouze-2010-RCF}. Another possible choice of a covariance model is a
power model, where the eigenvalues of the covariance are assumed to be a
negative power of $-\lambda_{i},i=1,\ldots,n$, e.g., \cite{Berner-2009-SSK,
Gaspari-2006-CAC,Simpson-2012-TCM}.

The difference of each estimator from the true matrix $D$ was measured in the
Frobenius norm again. To reduce the sampling noise, 50 replications have been
done for each sample size and the mean of squared Frobenius norm can be found
in Fig.~\ref{fig:comparison}.
%The difference of each estimator from the true matrix $D$ was measured in the
%Frobenius norm, which is the same as the Euclidean norm of a matrix written as
%one long vector. In order to reduce the sampling error, 50 simulations of the
%same size were generated and the mean Frobenius norm squared was computed. The
%results can be found in Fig.~\ref{fig:comparison}.

% fig 3
\begin{figure}[h]
\centering
\includegraphics[height=3in]{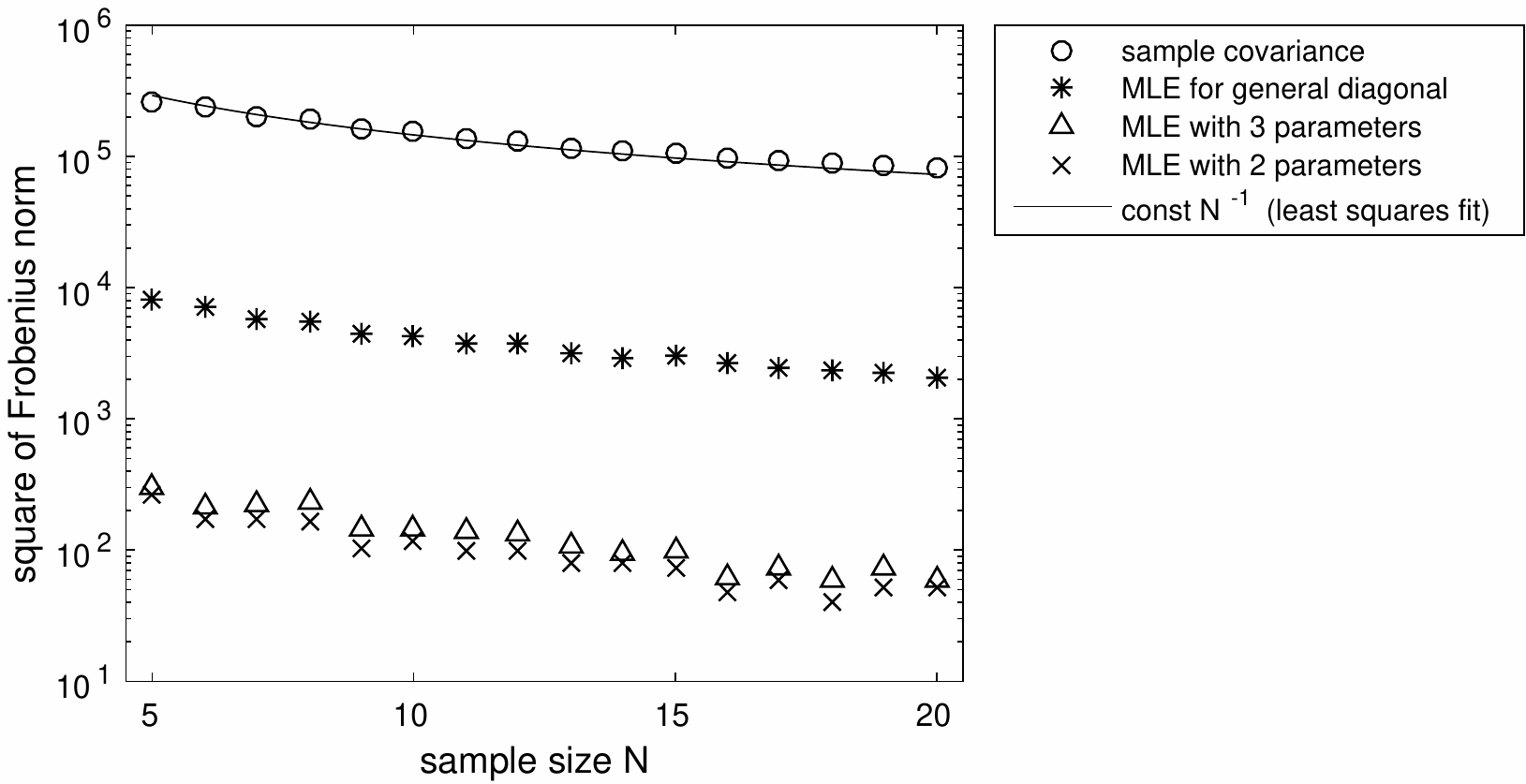} 
\caption{Comparison
of the error matrix $\hat{D}-D$ in the Frobenius norm. The field had dimension
$n=10\times10$. Exponential decay of eigenvalues (i.e. $\tau_{i}
=ce^{\alpha\lambda_{i}},i=1,\ldots,n$ ) was used with parameters $c=1/30$ and
$\alpha=0.002$. The full line is the order of convergence $\mathrm{const}%
(N^{-1})$ fitted to the error of the sample covariance. }%
\label{fig:comparison}%
\end{figure}

For the diagonal MLE, given by (\ref{D1}), (\ref{D03}), and (\ref{D02}), we
can expect from (\ref{eq:2-norm-comparison}) that these estimators should
satisfy asymptotically
\begin{equation}
\E\left(  \left\vert \hat{D}_{N}^{\left(  k\right)  }-D\right\vert _{F}%
^{2}\right)  \approx\frac{1}{N}\operatorname*{Tr}(J_{D^{\left(  k\right)  }%
}^{-1}),\quad k=1,2,3, \label{eq:simulation-mean-square}%
\end{equation}
even if convergence in distribution does not imply convergence of moments
without additional assumptions. This conjecture can be supported by a
comparison of Figures \ref{fig:com_Frob} and \ref{fig:com_trace}, where we
observe the same decay. From the nesting, we know that%
\begin{equation}
\operatorname*{Tr}(J_{D^{\left(  2\right)  }}^{-1})\leq\operatorname*{Tr}%
(J_{D^{\left(  3\right)  }}^{-1})\leq\operatorname*{Tr}(J_{D^{\left(
1\right)  }}^{-1}) \label{eq:simulation-trace-comparison}%
\end{equation}
and we can expect that the Frobenius norm should decrease for more restrictive
models, that is,%
\begin{equation}
\E\left\vert \hat{D}_{N}^{\left(  2\right)  }-D\right\vert _{F}^{2}%
\leq\E\left\vert \hat{D}_{N}^{\left(  3\right)  }-D\right\vert _{F}^{2}%
\leq\E\left\vert \hat{D}_{N}^{\left(  1\right)  }-D\right\vert _{F}^{2},
\label{eq:mean-squared-comp}%
\end{equation}
which is confirmed by the simulations (see Figure \ref{fig:com_trace}, resp.
\ref{fig:com_Frob}).

The comparisons (\ref{eq:mean-squared-comp}) of the Frobenius norm of the
error in the mean squared complement the pointwise comparison
(\ref{eq:pointwise}) between the sample covariance and its diagonal. Relying
on MLE\ for that comparison is not practical, because the sample size of
interest here is $N<n$, and, consequently, $\hat{\Sigma}_{N}$ is singular and
cannot be cast as MLE\ with an accompanying Fisher information matrix, cf.
Remark \ref{rem:singular}. But it is evident that for small sample sizes,
estimators computed in the proper subspace perform better. Hence, the
hierarchical order seems to hold even when $N<n$.

%fig4
\begin{figure}[ptb]
\centering
\includegraphics[height=3in]{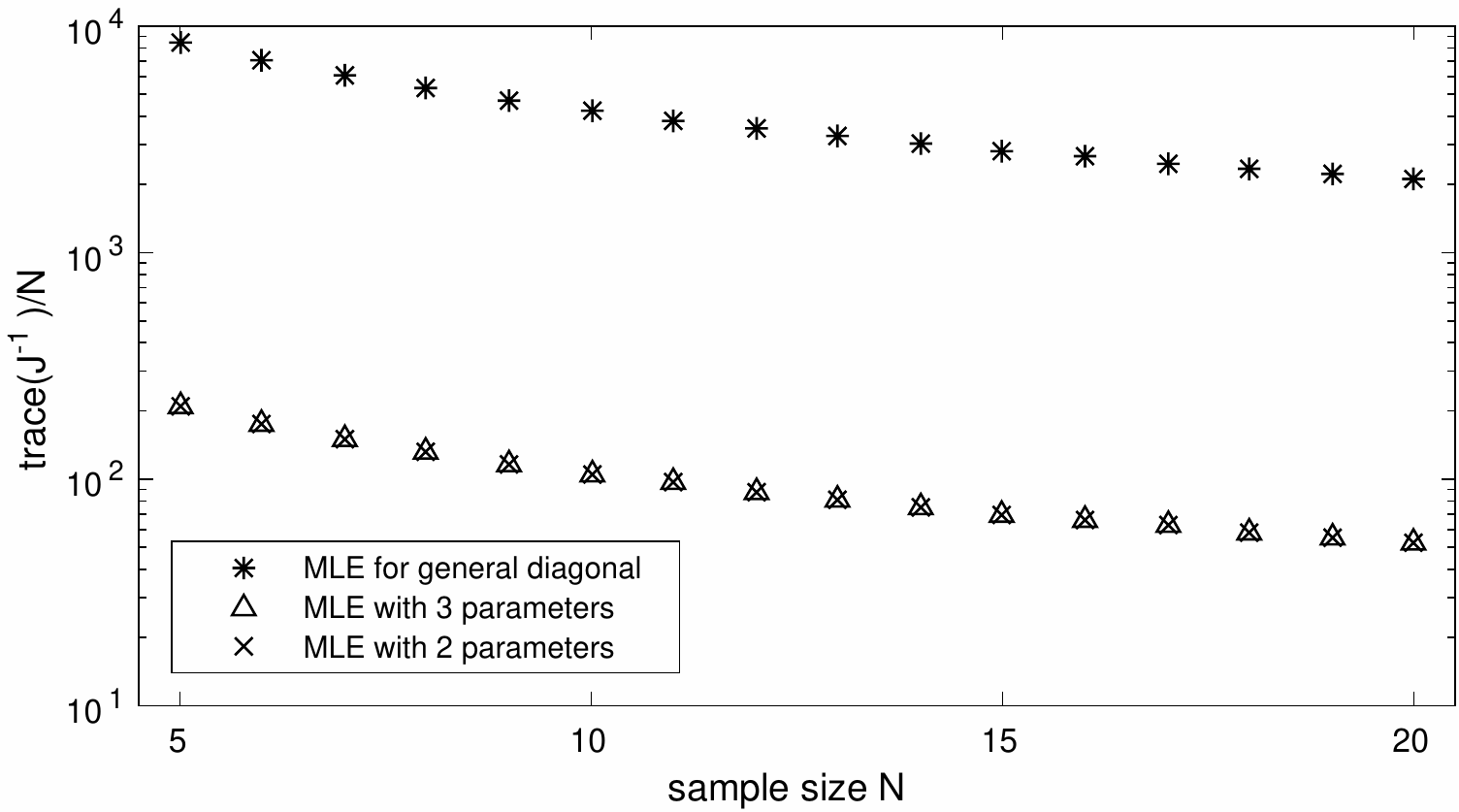}
\caption{$\frac{1}%
{N}\operatorname*{Tr}(J^{-1}_{D})$ for the three parameterizations. }%
\label{fig:com_trace}%
\end{figure}

%fig5
\begin{figure}[ptb]
\centering
\includegraphics[height=3in]{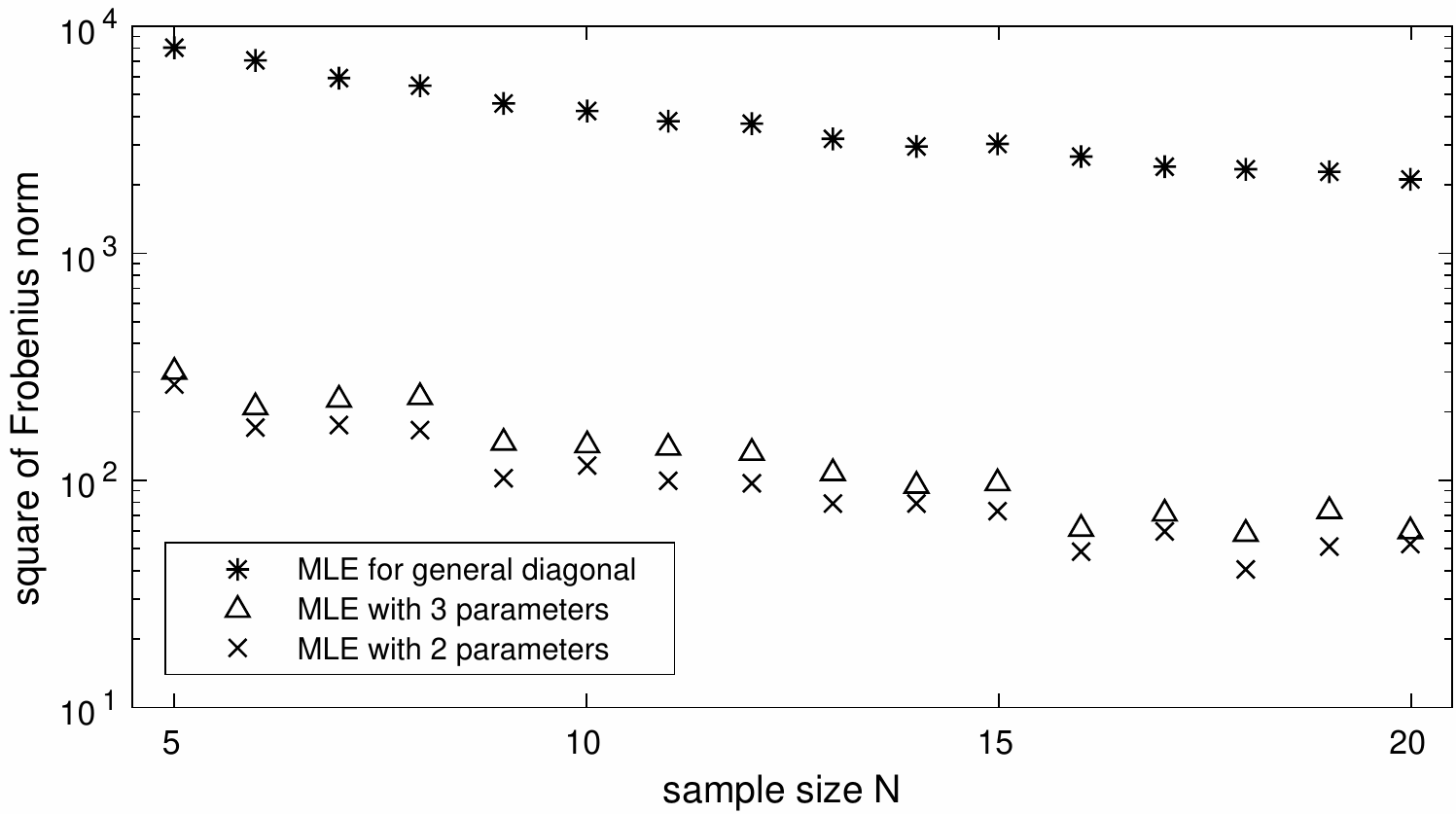}%
\caption{ Mean of $|
\hat{D}_{N}^{\left(  k\right)  }-D|_{F}^{2} $ based on 50 replications}%
\label{fig:com_Frob}%
\end{figure}

\section{COMPARISON WITH REGULARIZATION METHODS}

\label{sec:shrinkage}

In the previous sections, we pointed out the advantages of using
low-parametric models for estimating a covariance matrix using a small sample.
As mentioned in the Introduction, there is another large class of estimating
methods for high-dimensional covariance matrices: shrinkage estimators. The
principle of these methods is to move the sample covariance towards a target
matrix that possesses some desired properties (e.g., full rank, proper
structure). This can be seen as a convex combination of the sample covariance
matrix $\hat{\Sigma}_{N}$ and the so called target matrix $T$:
\begin{equation}
\hat{\Sigma}_{S}=\gamma\hat{\Sigma}_{N}+(1-\gamma)T,\hspace{5mm}\text{ for
}\gamma\in\lbrack0,1]. \label{general-shrink}%
\end{equation}

%However, shrinkage can be generally achieved in different way and we chose two distinct representatives.
%There exist a lot of shrinkage estimators so we chose two distinct representatives.
One of the simplest shrinkage estimators has the form of (\ref{general-shrink}%
) with the target matrix equal to identity, which results in shrinking all
sample eigenvalues with the same intensity towards their mean value.
\cite{Ledoit-2004-WEL} derived the optimal shrinkage parameter $\gamma$ to
minimize the squared Frobenius loss
\begin{equation}
\min_{\gamma}\E||\hat{\Sigma}_{S}-D||_{F}^{2}.
\end{equation}
The comparison of this estimator with the maximum likelihood estimator
$\hat{D}^{(2)}$ was accomplished by a simulation with identical setting as in
Section \ref{sec:comp}. The results are shown in Fig.~\ref{fig:com_shrink}.
For reference, the sample covariance $\hat{\Sigma}_{N}$ and its diagonal
$\hat{D}^{(0)}$ are also added.

Another regularization method is described in
%(\ref{general-shrink}).
\cite{Won-2013-CCE}. They consider a type of covariance estimator, where the
regularization effect is achieved by bounding the condition number of the
estimate by a regularization parameter $\kappa_{max}$. Since the condition
number is defined as a ratio of the largest and smallest eigenvalue, this
method corrects for overestimation of the largest eigenvalues and
underestimation of the small eigenvalues simultaneously. The resulting
estimator is called a condition-number-regularized covariance estimator and it
is formulated as the maximum likelihood estimator restricted on the subspace
of matrices with condition number bounded by $\kappa_{max}$, i.e.
\begin{equation}
\max_{\Sigma}\ell(\Sigma)\hspace{4mm}\text{ subject to }\frac{\lambda
_{max}(\Sigma)}{\lambda_{min}(\Sigma)}\leq\kappa_{max},
\end{equation}
where $\lambda_{max}(\Sigma)$, resp. $\lambda_{min}(\Sigma)$, is the largest,
resp. the smallest, eigenvalue of the covariance matrix $\Sigma$. An optimal
$\kappa_{max}$ is selected by maximization of the expected likelihood, which
is approximated by using $K$-fold cross-validation. The authors proved that
$\kappa_{max}$ selected in this way is a consistent estimator for the true
condition number (i.e. the condition number of $D$). Therefore, the idea of
this method is to search a MLE in a subspace defined by covariance matrices
with condition number smaller or equal to the true condition number. The form
of the resulting covariance estimator together with the details of the
computational process is provided in \cite{Won-2013-CCE}. In
Fig.~\ref{fig:com_shrink}, we can see the performance of this estimator
(denoted as \texttt{cond-num-regularization}) in comparison of other methods.

% fig6
\begin{figure}[ptb]
\centering
\includegraphics[height=3in]{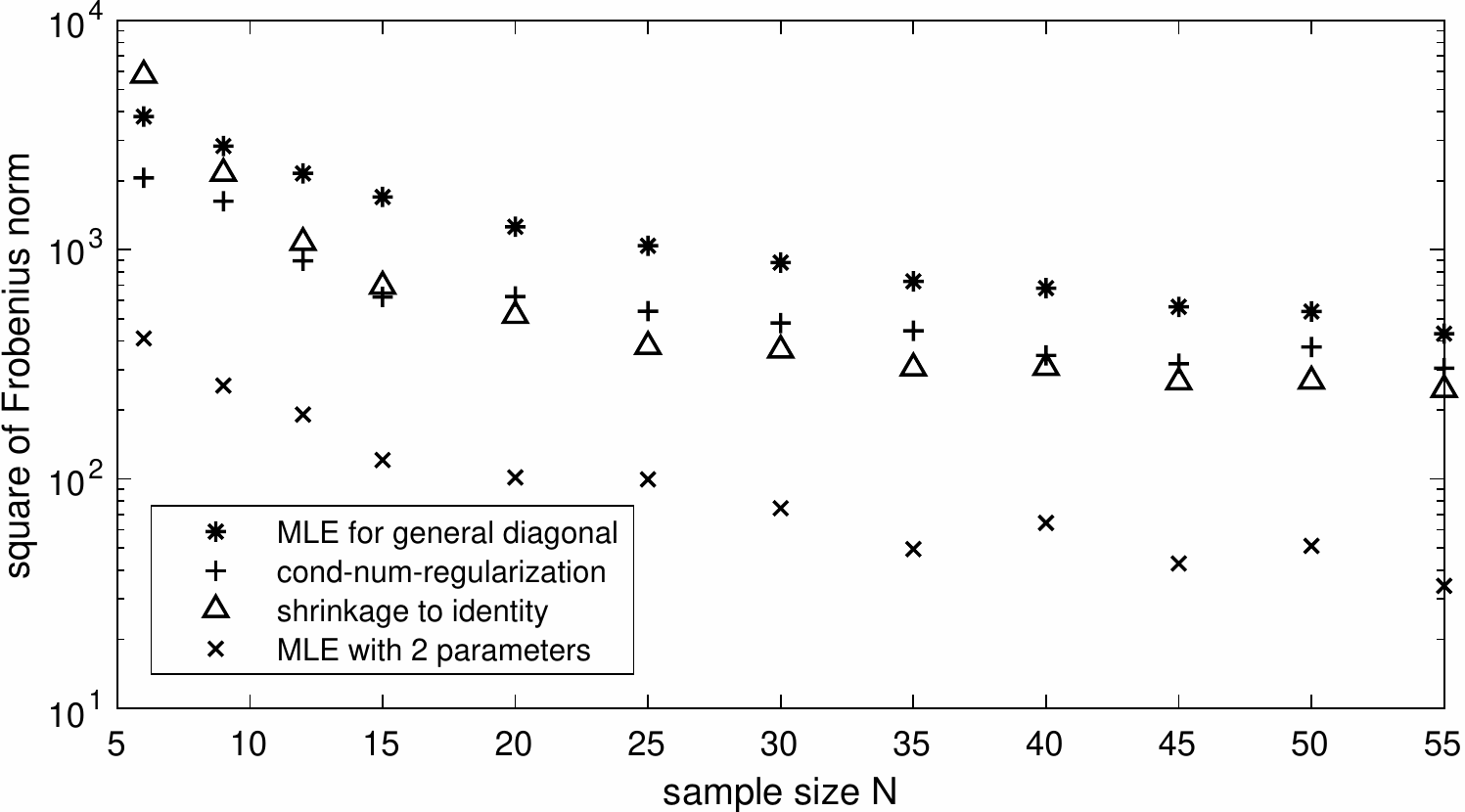}
\caption{
Comparison of regularization estimators with maximum likelihood estimators.
The error matrices $\hat{\Sigma}-D$ are compared in the Frobenius norm. The
simulation setting was identical with the Section \ref{sec:comp}.}%
\label{fig:com_shrink}%
\end{figure}

The shrinkage estimator $\hat{\Sigma}_{S}$ and the
condition-number-regularized estimator result in non-diagonal matrices, which
in our case predetermines them to perform worse than the diagonal estimator
$\hat{D}^{(0)}$. However, we have to note that performance of these methods
strongly depends on the particular form of the true covariance matrix $D$. In
the case when the decrease of the true eigenvalues is less rapid, both methods
may provide better results than the diagonal of sample covariance. The
performance of $\hat{\Sigma}_{S}$ could be possibly improved by choosing a
different target matrix that is closer to reality but such a study is out of
the scope of this paper.

It is seen from Fig.~\ref{fig:com_shrink} that the
condition-number-regularized estimator provides more precise estimates than
the sample covariance $\hat{\Sigma}_{N}$, as expected. This is in accordance
with the preceding theory and empirical findings about the higher precision of
estimators from a smaller parametric subspace (the corresponding parametric
subspace consists of matrices with the condition number smaller or equal to
$\kappa_{max}$). If, however, the theoretical condition number is very large
as in our case, the method has a problem in estimating this number and its
performance is limited.
%The shrinkage estimator $\hat{\Sigma}_{S}$ is not a maximum
%likelihood estimator but it makes the eigenspectrum be less spread than that
%of the sample covariance matrix, which leads to a substantial improvement.

Both regularization estimators perform well against sample covariance, but the
setting of our simulation is less favourable for them. Neither of them can
compete with the maximum likelihood estimator found in the true small subspace
of diagonal matrices with proper decay.

\section{CONCLUSIONS}

Our main aim was to point out the significant advantage resulting from
computing the MLE of the covariance matrix in a proper parameter subspace,
especially in the high-dimensional setting, when the available sample has
small size relative to the dimension of the problem. This subspace can be
formed, e.g., by a parametric model for covariance eigenvalues or for a
diagonal matrix resulting from a suitable set of transformations.

We provided theoretical results on asymptotic comparison of covariance
matrices of each estimator for multivariate normal distribution, where we can
lean on the well-developed maximum likelihood theory. The situation for small
samples was illustrated by means of a simulation. We consider two-parametric
models for the covariance eigenvalues based on the eigenvalues of Laplace
operator. In practice, the proper model/subspace can be inferred from
historical data.

Using a properly specified model, one can reach a significant improvement in
performance, which can have a positive impact on the subsequent tasks like
data assimilation and prediction.

\section*{ACKNOWLEDGEMENTS}

This work was partially supported by the the Czech Science Foundation (GACR)
under grant 13-34856S and by the U.S. National Science Foundation under 
grants DMS-1216481 and ICER-1664175.

\bibliographystyle{apacite}
\bibliography{references,../../references/geo,../../references/other}

\end{document}